\documentclass[11pt, leqno, letterpaper]{amsart}

\usepackage{graphicx}    
\usepackage{bbm, xcolor}

\usepackage[margin=1in]{geometry}

\usepackage[english]{babel}

\usepackage{amsmath,amsthm,amsfonts,amssymb}
\usepackage{algorithmic}
\usepackage{algorithm}

\newtheorem{proposition}{Proposition}[section]

\newtheorem{definition}{Definition}[section]
\newtheorem{remark}{Remark}[section]

\begin{document}

\title[The logistic-normal integral and the moments of the logistic-normal distribution]
{The logistic-normal integral and the moments of the logistic-normal distribution}

\author{Dan Pirjol}
\address{Stevens Institute of Technology, Hoboken, NJ 07030}

\date{August 2026}


\begin{abstract}
The logistic-normal integral appears in problems of statistical estimation 
for logistic models with Gaussian random effects, and generalized linear mixed models. 
We study the numerical evaluation of this integral and of its derivatives, and give
closed form evaluations at certain points and series expansions for the general case.
There is a continuum of possible series expansions, and 
we single out one series expansion which is optimal 
for numerical evaluation. 
We propose an algorithm for a precise numerical evaluation, 
based on the
optimal series, with good approximation error control in the tails. 
As an application we give explicit results for the 
first two moments of a logistic-normal random variable.
\end{abstract}

\maketitle

\baselineskip18pt

\section{Introduction}

The logistic-normal integral plays an important role in 
estimation problems for logistic models with Gaussian random
effects, mixed linear models, and generalized mixed linear models.
See \cite{Demidenko,Diggle2002,CSLN,Stefanski} for overviews of these
topics.  It can be defined as 
\begin{eqnarray}\label{varphidef}
\varphi(x,t) = \int_{-\infty}^\infty \frac{dy}{\sqrt{2\pi t}}
e^{-\frac{1}{2t}(x-y)^2} \frac{1}{1+e^y} \,.
\end{eqnarray}
The likelihood function for the
logistic model with multivariate Gaussian correlated variables can be reduced
to the evaluation of the integral $\varphi(x,t)$, see \cite{CSLN} and \cite{GSM}. 
Similar integrals appear also
when pricing discounted cash flows in interest rate models with
log-normally distributed rates, see \cite{HWQF} for such an application.

The logistic-normal integral cannot be expressed in closed form, and 
several approximate methods have been proposed for its numerical evaluation, 
see Chapter~7.12 in \cite{Demidenko} and Section~18.4 in \cite{Balakrishnan} 
for surveys. 
A precise determination has been proposed by  \cite{CS} 
using trapezoidal quadrature on a uniform grid. The quadrature error is 
bounded using the method proposed in \cite{Goodwin}, see \cite{TW} for an 
overview of the method. A good approximation 
as a mixture of normal cumulative distributions was proposed by \cite{StefanskiMix} and \cite{MonahanStefanski}.

In a previous paper \cite{DP2} it was noted that the logistic-normal integral is a particular
case of the Mordell integral introduced in \cite{Mordell1933}, which 
has been studied extensively in analytical number theory and in relation to
mock theta functions \cite{Zwegers}. 
We explore further in this paper the implications of the theory of the Mordell
integral for the numerical evaluation of the logistic-normal integral. 
This satisfies several symmetries relations which aid
considerably with its numerical evaluation, leading to closed form evaluations at
certain points and convergent series expansions for the general case.

In Section \ref{sec:2} we summarize the symmetry relations satisfied by the
logistic-normal integral, and its exact evaluations at certain points. 
Section~\ref{sec:3} presents series expansions for the logistic-normal integral,
obtained using the Poisson summation formula. 
In Proposition~\ref{propPoisson} we present two equivalent series expansions,
following from the Poisson summation formula.
Proposition~\ref{prop:gen} shows that there is a continuum of such series expansions, 
indexed by one complex parameter. 

Section~\ref{sec:4} discusses the evaluation of a class of generalizations of the
logistic-normal integral that are related to the derivatives of the logistic-normal integral with respect to its parameters. 
We give explicit results for the exact evaluation of the first few generalized integrals on grids of uniformly spaced points, which can be used to construct interpolations for these functions.
As an application we give in Section~\ref{sec:5} explicit expressions 
for the first four moments of a logistic-normal random variable.

In Section~\ref{sec:6} we present details of implementation of these series and 
numerical tests of their efficiency.
One particular series, given in equation (\ref{PoissonTheta2}),
turns out to be optimally suited for this purpose, and in Appendix \ref{sec:C}
we present an algorithm 
for a precise numerical evaluation of the logistic-normal integral based on it.
We present numerical tests of the
approximation proposed which demonstrate good agreement with benchmark
evaluations based on the trapezoidal quadrature method of \cite{CS}.
The performance of the series expansion is compared with that of the interpolation 
from exactly known values, on the example of the first two moments of a logistic-normal integral.
In order to make the discussion self-contained, we give in 
Appendix A a brief survey of relevant results on the theory of the Mordell integral. 
Appendix B contains the proofs. 

\section{Exact evaluations and a symmetry relation}
\label{sec:2}

It is convenient to work with the function $g(x,t)$ defined as
\begin{eqnarray}\label{gdef}
g(x,t) = \int_{-\infty}^\infty \frac{dy}{\sqrt{2\pi t}}
e^{-\frac{1}{2t}(x-y)^2} \frac{1}{\cosh(y/2)}\,.
\end{eqnarray}
The logistic-normal integral is expressed in terms of this function as
\begin{eqnarray}\label{varphi2g}
\varphi(x,t)=\frac12 e^{-\frac12 x + \frac18 t}  g(x-\frac12 t,t) \,.
\end{eqnarray}

The function $g(x,t)$ satisfies the recursion relation
\begin{eqnarray}\label{Mordell2g}
e^{-\frac12 x} g\left(x-\frac12 t,t\right) + 
e^{\frac12 x}  g\left(x+\frac12 t,t\right) = 
2e^{-\frac18 t}
\end{eqnarray}
which follows from the recursion relation for $\varphi(x,t)$ \cite{DP}
\begin{eqnarray}\label{LNrec}
\varphi(x+t, t) = e^{-x-\frac12 t} (1-\varphi(x,t))\,.
\end{eqnarray}
Taking  $x=0$ in the relation (\ref{Mordell2g}) gives
\begin{eqnarray}\label{gatt2}
&& g\left(\pm\frac12 t, t\right) = e^{-\frac18 t}\,.
\end{eqnarray}

Repeated application of (\ref{Mordell2g}) starting with (\ref{gatt2}) gives exact evaluations of the function $g(x,t)$ on the grid of values $x_k = (k+\frac12)t$ with $k\in \mathbb{Z}$.
This can be expressed equivalently as an exact evaluation of the logistic-normal integral
$\varphi(x_k,t)$ on the grid $x_k=kt$. 

\begin{proposition}\label{prop:2}
For all $k\geq 0$ we have
\begin{align}
\varphi(kt,t) &= e^{-\frac12 k^2 t} 
\sum_{j=0}^{k-2} (-1)^j e^{\frac12(j-k+1)^2t} + \frac12 (-1)^{k-1} e^{-(k-1)^2 t}\,, \quad k\geq 2 \\
&= \frac12 e^{-\frac12 t} \,, \hspace{7.3cm} k = 1 \nonumber \\
& = \frac12 \,,\hspace{8cm} k = 0 \,.\nonumber
\end{align}
The $k<0$ values are obtained from $\varphi(-kt,t) = 1- \varphi(kt,t)$.
\end{proposition}

\begin{proof}
The proof is given in Appendix~\ref{sec:app.B}.
\end{proof}



We give next another symmetry relation of the function $g(x,t)$ which
relates its values for real and imaginary values of $x$ at different $t$.
This will be useful in order to express the series expansions for $g(x,t)$
in an alternative form.

\begin{proposition}\label{prop:1}
The function $g(x,t)$ satisfies the relation 
\begin{eqnarray}\label{Mordell5g}
g(x,t) = \sqrt{\frac{2\pi}{t}} e^{-\frac{1}{2t}x^2}
g\left(\frac{2\pi x}{it}, \frac{4\pi^2}{t}\right)\,.
\end{eqnarray}
\end{proposition}

\begin{proof}
The proof is given in Appendix~\ref{sec:app.B}.
\end{proof}


\section{Series expansions for the logistic-normal integral}
\label{sec:3}

Series expansions for the Mordell integral have been obtained by an
application of the Poisson summation formula in \cite{Mordell1933}.
In the context of the numerical evaluation of the logistic-normal integral
one such series was obtained for the $g(z,t)$ function by the same method 
in \cite{DP}, exploiting the quasi-periodicity relation (\ref{Mordell2g}). 
The result in \cite{DP} can be expressed in two alternative forms,
one of which will turn out to be more suited for the numerical evaluation of
this function. 

\begin{proposition}\label{propPoisson}
i) The function $g(z,t)$ has the series expansion
\begin{eqnarray}\label{PoissonTheta4}
&& \vartheta_4(\frac{i}{2}z, e^{-\frac12 t})
g(z,t)  = 
2 \sum_{n=-\infty}^\infty (-1)^n e^{ (n-\frac12) z}
\frac{q^{n^2-\frac14}}{1 - q^{2n-1}}
+ \frac{4\pi}{t}
\sum_{n=-\infty}^\infty  \exp\Big( \frac{2\pi in z}{t} \Big)
\frac{q_1^{n^2+n}}{1 + q_1^{2n}}  \,.
\end{eqnarray}

ii) An equivalent series expansion is given by
\begin{eqnarray}\label{PoissonTheta2}
&& \vartheta_2(\frac{i}{2}z, e^{-\frac12 t})
g(z,t) = 
\frac{4\pi}{t} \sum_{n=-\infty}^\infty (-1)^n e^{-\frac{2\pi i}{t} (n-\frac12) z}
\frac{q_1^{n^2-\frac14}}{1 - q_1^{2n-1}}
+ 2
\sum_{n=-\infty}^\infty  e^{nz}
\frac{q^{n^2+n}}{1 + q^{2n}}  \,.
\end{eqnarray}
In both expressions we denoted $q = e^{-\frac12 t}, q_1 = e^{-\frac{2\pi^2}{t}}$.
\end{proposition}

\begin{proof}
The proof is given in the Appendix B.
\end{proof}

Although mathematically equivalent, these two series have very different 
properties when used for the numerical evaluation of $g(z,t)$, as they involve 
a ratio of two functions. 
The Jacobi theta function $\vartheta_4(i\frac{z}{2},e^{-\frac12 t})$ has zeros 
at the points
$z_{m,n} = \frac12 t + m 2\pi i  + n t$ with $(m,n)\in \mathbb{Z}$.
In particular, $\vartheta_4(\frac{i}{2}z,e^{-\frac12 t})$ vanishes at the points 
$z = \pm \frac12 t$, such that at these  points 
the left-hand
side of (\ref{PoissonTheta4}) vanishes. The right-hand side of 
(\ref{PoissonTheta4}) must 
also vanish at this point, because $g(\frac12 t,t)$ is finite, see 
Eq.~(\ref{gatt2}).
As noted in \cite{DP}, this introduces numerical errors in the evaluation
of $g(z,t)$ close to the points $z_k=(k+\frac12)t$, where the expression
(\ref{PoissonTheta4}) leads to ratios of very small numbers.

Consider next the equivalent series (\ref{PoissonTheta2}).
The function $\vartheta_2(i\frac{z}{2},e^{-\frac12 t})$ has simple zeros at
$z_{m,n} = -i\pi + m 2\pi i  + n t$ with $ (m,n)\in \mathbb{Z}$.
None of these points is on the real axis of $z$. The zeros which are closest
to the real axis are at $z_k = \pm  i\pi + kt$. This avoids the 
numerical instabilities which are introduced in the numerical evaluation of 
$g(z,t)$ using the series expansion (\ref{PoissonTheta4})
by the 0/0 limit around the $z = \pm\frac12 t$ points. 
Numerical tests presented in Sec.~\ref{sec:6} confirm that the relation 
(\ref{PoissonTheta2}) gives a more stable
numerical evaluation of $g(z,t)$ along the real axis $z\in \mathbb{R}$.


The series expansions in Proposition~\ref{propPoisson} can be expressed in 
a compact way in terms of the level one Appell-Lerch sums.
They are defined as \cite{Appell}
\begin{eqnarray}
A_1(u,v;\tau) = e^{\pi iu} 
\sum_{k=-\infty}^\infty
(-1)^k \frac{e^{\pi i \tau (k^2+k)} e^{2\pi i k v}}
{1 - e^{2\pi i k\tau + 2\pi i u}} = z^\frac12
\sum_{k=-\infty}^\infty (-1)^k \frac{q^{ k(k+1)}}{1-z q^{2k}} y^k
\end{eqnarray}
where $z = e^{2\pi i u}, y = e^{2\pi i v}, q = e^{\pi i \tau}$ and
$v \in \mathbb{C}, u \in \mathbb{C} \backslash ( \mathbb{Z} \tau + \mathbb{Z} )$.

These functions satisfy a large number of symmetry relations, see 
Proposition 1.4 in \cite{Zwegers}.
These functions are related to $\mu(u,v;\tau)$ defined in
Proposition~1.4 of \cite{Zwegers} as
$A_1(u,v;\tau) = \mu(u,v;\tau) \vartheta_1(v;\tau)$.


We give next a general series expansion for $g(z,t)$ in terms of the 
Appell-Lerch sums.
\begin{proposition}\label{prop:gen}
The function $g(z,t)$ has the series expansion 
\begin{eqnarray}\label{vgeneral}
&& \frac{i}{2} \vartheta_1(v\pi;e^{-\frac12 t}) g(z,t) = \\
&& \qquad A_1\left(v+\frac{iz}{2\pi},v;\frac{it}{2\pi}\right) + 
\frac{2\pi i}{t} e^{- \frac{z^2}{2t} -\frac{2\pi^2 v^2}{t}}
A_1\left(\frac{z}{t} + \frac{2\pi v}{it}, \frac{2\pi v}{it} ; 
\frac{2\pi i}{t}\right) \,, \nonumber
\end{eqnarray}
where $v$ is an arbitrary complex number.
\end{proposition}

\begin{proof}
The Appell-Lerch sums are related to the Mordell integral $h(u;\tau)$
by modular transformations, 
see Proposition 1.5 in \cite{Zwegers}
\begin{eqnarray}\label{A1mod}
A_1(u,v;\tau)  - \frac{1}{\tau} e^{\frac{\pi i}{\tau} (u^2-2u v)}
A_1\left(\frac{u}{\tau}, \frac{v}{\tau} ; - \frac{1}{\tau}\right) = 
-\frac{i}{2} \vartheta_1(v\pi;\tau) h(u-v;\tau) \,.
\end{eqnarray}

Take $u-v=\frac{iz}{2\pi}$ and $\tau=\frac{it}{2\pi}$ 
in (\ref{A1mod}). Next we note that we have the relation
\begin{eqnarray}\label{g2hp}
g(z,t) = h\left( \frac{iz}{2\pi}, \frac{it}{2\pi}\right)\,,
\end{eqnarray}
which is obtained by 
combining the relations (\ref{Mordell5g}) and (\ref{g2h}).
Using (\ref{g2hp}) yields (\ref{vgeneral}).
\end{proof}

The relation (\ref{vgeneral}) is the most general series expansion for $g(z,t)$. 
There are infinitely many such expressions, since $v$ is arbitrary. 
We would like to ask whether there is an optimal
choice for $v$ for the purpose of the numerical evaluation of $g(z,t)$. 

We note two possible choices for $v$ which simplify the
expression (\ref{vgeneral}). They are motivated by the 
observation that $v$ appears in the arguments of the 
Appell-Lerch sum in (\ref{vgeneral}) as $v$
and in the combination $v+\frac{iz}{2\pi}$. 
Under the following choices either is a constant and does not depend on $z$.

\begin{remark}
Consider the two choices for $v$ in (\ref{vgeneral})
\begin{eqnarray}\label{choice1}
&& i) \qquad v = - \frac{i}{2\pi} (z + C) \\
\label{choice2}
&& ii) \qquad v = - \frac{i}{2\pi} C
\end{eqnarray}
with $C$ a constant. It is easy to see that
with the substitution i) the choice $C=i\pi$ reproduces 
the series (\ref{PoissonTheta2}), and $C=\frac12 t$ reproduces the series 
(\ref{PoissonTheta4}).
\end{remark}

The choice ii) with $C=i\pi$ gives the following series
\begin{eqnarray}
\vartheta_2(0, q) g(z,t) &=& 2e^{-\frac12 z} 
\sum_{k=-\infty}^\infty \frac{q^{k^2+k}}{1+q^{2k} e^{-z}} \\
&+& \frac{4\pi}{t} e^{\frac{\pi i z}{t} - \frac{z^2}{2t}}
\sum_{k=-\infty}^\infty \frac{q_1^{k^2-\frac14}}{1+q_1^{2k-1} e^{\frac{2\pi i}{t}z}}
\nonumber
\end{eqnarray}
and with $C=\frac12 t$ 
\begin{eqnarray}
\vartheta_4(0, q) g(z,t) &=& 2e^{-\frac12 z} 
\sum_{k=-\infty}^\infty (-1)^k \frac{q^{k^2 - \frac14 }}{1-q^{2k-1} e^{-z}} \\
&+& \frac{4\pi}{t} e^{\frac{\pi i z}{t} - \frac{z^2}{2t}}
\sum_{k=-\infty}^\infty \frac{q_1^{k^2+k}}{1+q_1^{2k} e^{\frac{2\pi i}{t}z}}
\nonumber
\end{eqnarray}
with $q=e^{-\frac12 t}, q_1=e^{-\frac{2\pi^2}{t}}$. These series are very
similar to (\ref{PoissonTheta2}), (\ref{PoissonTheta4}), although the evaluation
of the second sum in either expression requires complex arithmetic for any $z$, 
which may be inconvenient for the evaluation of $g(z,t)$ for $z\in \mathbb{R}$.


We will show next that the series (\ref{PoissonTheta2}) corresponds to 
an optimal choice of $v$ from the point of view of the numerical evaluation
of the function $g(z,t)$ with $z\in\mathbb{R}$.
For the purpose of the numerical evaluation of $g(z,t)$, we would like
to choose $v$ such that the evaluation of this function from (\ref{vgeneral})
does not involve the ratio of two small numbers. The Jacobi theta function on 
the left-hand side $\vartheta_1(v\pi, e^{-\frac12 t})$ has zeros at 
$\pi v_{m,n} = m\pi + in \frac{t}{2}$ with $m,n\in \mathbb{Z}$.
If $v$ is chosen to be a function of $z$, as in (\ref{choice1}), then we would 
like the zeros of $i\vartheta_1(\pi v(z,t), e^{-\frac12 t})$ in the $z$ 
plane to be as far away as possible from the real axis. 
For the choice (\ref{choice1}) the zeros
are at $z_{m,n}+C = 2i\pi m - n t$, so choosing $C=i\pi$ ensures that the zeros
are as far as possible from the real axis (the nearest zeros are at $\pm i\pi$). 
Using the relation (\cite{WW}, p.~464) 
\begin{eqnarray}
\vartheta_1\left(-\frac{iz}{2} + \frac{\pi}{2}, e^{-\frac12 t}\right)= 
\vartheta_2\left(\frac{iz}{2}, e^{-\frac12 t}\right)
\end{eqnarray}
one finds that (\ref{vgeneral}) reproduces the series (\ref{PoissonTheta2}).
This argument suggests that this optimal choice of $v$
ensures numerical stability in numerical evaluations of $g(z,t)$ with $z\in\mathbb{R}$.

\section{Derivatives of the logistic-normal integral}
\label{sec:4}

Certain statistics applications require the evaluation of the 
derivatives of the logistic-normal integral. The derivatives of the 
logistic-normal integral are relevant for the solution of the
maximum likelihood estimation problem for a logistic model with Gaussian noise.

We consider in this section a class of integrals which are useful for this purpose,
and present exact evaluations on discrete grids of $z$ points and a series expansion.
Define
\begin{equation}
\varphi_j(z,t) = 
\int_{-\infty}^\infty \frac{x^j}{1+e^x} e^{-\frac{1}{2t}(x-z)^2}
\frac{dx}{\sqrt{2\pi t}}\,.
\end{equation}
For $j=0$ the function $\varphi_0(z,t)$ reproduces the logistic-normal 
integral (\ref{varphidef}).

These integrals are related to the derivatives of the logistic-normal integral with respect to its first argument. The first few derivatives are given below, for $j = 0,1,2, \cdots$
\begin{eqnarray}\label{dphi}
&& \partial_z \varphi_j(z,t) = - \frac{z}{t} \varphi_j(z,t) + \frac{1}{t}\varphi_{j+1}(z,t) \\
&& \partial_z^2 \varphi_j(z,t) = \Big( \frac{z^2}{t^2} - \frac{1}{t}\Big) \varphi_j(z,t) 
-\frac{2z}{t^2} \varphi_{j+1}(z,t) + \frac{1}{t^2}\varphi_{j+2}(z,t) \,.
\end{eqnarray}
The derivative with respect to $t$ is obtained by noting that $\varphi_j(z,t)$ satisfy 
the 1-dimensional heat equation $\partial_t \varphi_j(z,t) = \frac12 \partial_z^2 \varphi_j(z,t)$.

It is convenient to introduce the functions
\begin{equation}
g_j(z,t) = \int_{-\infty}^\infty \frac{x^j}{\cosh(x/2)} e^{-\frac{1}{2t} (x-z)^2} \frac{dx}{\sqrt{2\pi t}}\,,\quad j=1,2,\cdots
\end{equation}
which are related to the integrals $\varphi_j(z,t)$ as
\begin{equation}
\varphi_j(z,t) = \frac12 e^{-\frac12 z + \frac18 t} g_j(z-\frac12 t,t) \,.
\end{equation}

The functions $g_j(z,t)$ and their evaluation were studied in Sec.~5 of \cite{DP}.
The functions $\varphi_j(z,t)$ satisfy recursion relations
\begin{equation}\label{varphijrec}
\varphi_j(z+t,t) = e^{-z-\frac12 t} (f_j(z,t) - \varphi_j(z,t))
\end{equation}
with $f_j(z,t) := \int_{-\infty}^\infty x^j e^{-\frac{1}{2t}(x-z)^2} \frac{dx}{\sqrt{2\pi t}}$.
They can be evaluated in closed form on grids of uniformly spaced real $z$,
which are integers (half-integers) of $t$ for even (odd) $j$. Table \ref{tab:1}
summarizes the information required for this evaluation for $j=0,1,2,3$.

The functions $g_j(z,t)$ are even (odd) in $z$ for even (odd) index $j$. This gives the following transformation of $\varphi_j(z,t)$ under a sign change of the first argument
\begin{align}\label{parity}
\varphi_j(-z,t) = \left\{
\begin{array}{cc}
e^{z+\frac12 t} \varphi_j(z+t,t) = f_j(z,t) - \varphi_j(z,t)  &  \mbox{ for even } j \\
- e^{z+\frac12 t} \varphi_j(z+t,t) = \varphi_j(z,t) - f_j(z,t) & \mbox{ for odd } j \\
\end{array}
\right.
\end{align}

We start by considering the evaluation of $\varphi_1(z,t)$, the first generalized logistic-normal integral.
As seen from \eqref{dphi} this function is related to the derivative of the logistic-normal integral with 
respect to its first argument
$\partial_z \varphi(z,t) = - \frac{z}{t} \varphi(z,t) + \frac{1}{t} \varphi_1(z,t)$.

\begin{table}
\begin{center}
\caption{Inputs for the exact evaluations of the first few logistic-normal integrals 
$\varphi_j(z,t)$.
These integrals can be evaluated exactly on the respective grids of points $z_k$ using the recursion relation (\ref{varphijrec}), starting with the exact value in the third 
column.} \label{tab:1}
{\begin{tabular}{lccc}
$j$	& Exact evaluation grid & Exact values & $f_j(x,t)$ \\ 
\hline
\hline
0 & $z_k = k t$ & $\varphi(0,t) = \frac12$ & $f_0(x,t)=1$ \\
1 & $z_k = (k+\frac12) t$ & $\varphi_1(\frac12 t,t) = 0$ & $f_1(x,t)=x$ \\
2 & $z_k = k t$ & $\varphi_2(0,t) = \frac12 t$ & $f_2(x,t)=x^2+t$ \\
3 & $z_k = (k + \frac12) t$ & $\varphi_3(\frac12 t,t) = 0$ & $f_3(x,t)=x^3+ 3xt$ \\
\hline
\end{tabular}}
{}%
\end{center}
\end{table}

The following result is an analog of Proposition \ref{prop:2} and gives
exact evaluations of $\varphi_1(z_k, t)$ 
at $z_k = (k + \frac12)t$. This can be proved either in a similar way to Proposition~\ref{prop:2}
or directly from the recursion \eqref{varphijrec}. For simplicity we omit the proof.

\begin{proposition}\label{prop:phi1}
For any $k\geq 0$ we have
\begin{align}
\varphi_1\Big((k+\frac12)t , t\Big) &= 
\sum_{j=0}^{k-1} (-1)^{k-j+1} (j+\frac12) t e^{\frac12 (j-k)(j+k+1)t} \,,\quad
k \geq 1 \\
&= 0 \,,\hspace{6.2cm} k=0 \,.\nonumber
\end{align}
The values for $k<0$ are obtained from \eqref{parity} as
\begin{equation}
\varphi_1\Big((-k+\frac12) t,t\Big) = - e^{kt} \varphi_1\Big((k+\frac12)t , t\Big) \,.
\end{equation}
\end{proposition}

The first few values are $\varphi_1(3/2 t, t) = \frac12 t e^{-t}, \varphi_1(5/2t, t) = \frac32 t e^{-2t}-\frac12t e^{-3t}$, $\varphi_1(7/2t,t)=\frac52 t e^{-3t}- \frac32 t e^{-5t}+\frac12 t e^{-6t}$.

We give next a series expansion for $g_1(z,t)$, similar to that for $g(z,t)$,
which can be used for numerical evaluation of $\varphi_1(z,t)$. 
We use for this purpose the series expansion (\ref{PoissonTheta2}) which can be 
written compactly as
\begin{eqnarray}\label{PoissonTheta2a}
&& \vartheta_2(\frac{i}{2}z, e^{-\frac12 t})
g(z,t)  = 
\frac{4\pi}{t} S_1(z,q_1)
+ 2 S_2(z,q)
\end{eqnarray}
with $q = e^{-\frac12 t}, q_1 = e^{-\frac{2\pi^2}{t}}$ and
\begin{eqnarray}
&& S_1(z,q_1) := \sum_{n=-\infty}^\infty (-1)^n e^{-\frac{2\pi i}{t}(k-\frac12) z}
\frac{q_1^{n^2-1/4}}{1-q_1^{2n-1}} \\
&& S_2(z,q) := \sum_{n=-\infty}^\infty  e^{nz}
\frac{q^{n^2+n}}{1 + q^{2n}}
\end{eqnarray}

The sum $S_1(z,q_1)$ is a Laurent sum in powers of $e^{-\frac{2\pi i}{t} z}$, 
and $S_2(z,q)$ is a Laurent sum in powers of $e^z$. 
It is easy to check that for any $|q_1|<1$ and $|q|<1$ they both converge within the annulus of convergence $0 \leq |z| < \infty$. 
Thus they can be differentiated term by term with respect to $z$, and define 
convergent series expansions for $S'_1(z,q_1), S'_2(z,q)$.

\begin{proposition}\label{prop:g1Poisson}
The function $g_1(z,t)$ has the series expansion
\begin{eqnarray}\label{g1Poisson}
&& \vartheta_2(\frac{i}{2}z,q) g_1(z,t) = 
4\pi \Big( S'_1(z,q_1) - R_2(z,t) S_1(z,q_1) + \frac{z}{t} S_1(z,q_1) \Big) \\
&& \hspace{2.6cm} + \,
2t \Big( S'_2(z,q) - R_2(z,t) S_2(z,q) + \frac{z}{t} S_2(z,q)\Big)\nonumber
\end{eqnarray}
with 
\begin{equation}\label{R2def}
R_2(z,t) := \frac{i}{2} \frac{\vartheta'_2(\frac{i}{2} z,q)}{\vartheta_2(\frac{i}{2} z,q)} =
\frac12 \tanh z + \sinh z \sum_{k=1}^\infty \frac{1}{\cosh(k t) + \cosh z}
\end{equation}
and $q=e^{-\frac12 t},q_1=e^{-2\pi^2/t}$.
\end{proposition}

\begin{proof}

Taking a derivative of the Poisson sum (\ref{PoissonTheta2}) with respect to $z$ gives 
\begin{eqnarray}\label{dzg}
&& \frac{i}{2} \vartheta'_2(\frac{i}{2}z,q) g(z,t) + \vartheta_2(\frac{i}{2}z,q) \partial_z g(z,t) \\
&& = R_2(z,t) \vartheta_2(\frac{i}{2}z,q) g(z,t) + \vartheta_2(\frac{i}{2}z,q) \partial_z g(z,t) \nonumber \\
&& = \frac{4\pi}{t} S'_1(z,q_1) + 2 S'_2(z,q) \,. \nonumber 
\end{eqnarray}
The series expansion for $R_2(z,t)$ defined in \eqref{R2def} is given in Problem 15 in 
Whittaker and Watson (1927), see page 489.

The equation (\ref{dzg}) can be expressed as a series expansion for $g_1(z,t)$ 
using the equation
\begin{equation}
g_1(z,t) = t \partial_z g(z,t) + z g(z,t) 
\end{equation}
which follows from \eqref{dphi} with $j=0$.
The final result can be put into the form (\ref{g1Poisson}).
\end{proof}

We give next exact evaluations also for the higher order logistic-normal integrals with $j=2,3$, which follow from the recursion relation \eqref{varphijrec}.

\begin{proposition}\label{prop:phi2}
For any $k\geq 0$ we have
\begin{align}
\varphi_2 (k t , t) &= (-1)^k e^{-\frac12 k^2 t}
\Big( \frac12 t - \frac12 
\sum_{j=0}^{k-1} (-1)^{j} t(1+j^2 t) e^{\frac12 j^2 t} \Big) \,,\quad
k \geq 1 \\
&= \frac12 t \,,\hspace{7.5cm} k=0 \,.\nonumber
\end{align}
The values for $k<0$ are obtained from \eqref{parity} as
\begin{equation}
\varphi_2( -k t, t ) = 
- e^{(k+\frac12) t} \varphi_2\Big((k+1)t, t\Big) \,.
\end{equation}
\end{proposition}

The first few evaluations are $\varphi_2(t, t) = \frac12 t e^{-\frac12 t}, 
\varphi_2(2t, t)=t(1+t) e^{-\frac32 t} - \frac12 t e^{-2t}$.

\begin{proposition}\label{prop:phi3}
For any $k\geq 0$ we have
\begin{align}
\varphi_3\Big((k + \frac12) t , t\Big) &= \frac18 (-1)^{k+1} e^{-\frac12 k(k+1) t} 
\sum_{j=0}^{k-1} (-1)^{j} t^2 (1+2j) (12+t + 4t j(j+1)) e^{\frac12 j(j+1) t}  \,,\quad
k \geq 1 \\
&= 0 \,,\hspace{11.5cm} k=0 \,.\nonumber
\end{align}
The values for $k<0$ are obtained from \eqref{parity} as
\begin{equation}
\varphi_3\Big((-k+\frac12) t, t\Big) = - e^{kt} \varphi_3\Big((k+\frac12)t , t\Big) \,.
\end{equation}
\end{proposition}

The first few evaluations are $\varphi_3(\frac32t, t) = \frac18 t^2 e^{- t} (t+12), 
\varphi_3(\frac52t, t)=- \frac18 t^2(12+t) e^{-3t} + 3 t^2(12 + 9t) e^{-t} $.


\section{Relation to logistic-normal random variables}
\label{sec:5}

The random variable $X$ has a logistic-normal distribution $X \sim logitnorm(\mu,\sigma)$ if it has the form $X = \frac{1}{1+e^{-Z}}$ with $Z\sim N(\mu,\sigma)$
a normally distributed random variable with mean $\mu$ and standard deviation $\sigma$. The moments of $X$ can be expressed in terms of the 
integrals $\varphi_j(z,t)$. We give below the explicit results for the first few moments.

\begin{proposition}
Define $X \sim logitnorm(\mu,\sigma)$.
We have
\begin{align}\label{mom1}
\mathbb{E}[X] &= \varphi(-\mu,\sigma^2) \\
\label{mom2}
\mathbb{E}[X^2] &= 
(1 + \frac{\mu}{\sigma^2})  \varphi(-\mu, \sigma^2) 
+ \frac{1}{\sigma^2} \varphi_1(-\mu, \sigma^2) \,,\\
\label{mom3}
\mathbb{E}[X^3] &= 
\Big( (1 + \frac{\mu}{\sigma^2}) (1 + \frac{\mu}{2\sigma^2}) - \frac{1}{2\sigma^2}\Big) 
\varphi(-\mu, \sigma^2) \\
&+ \frac{1}{\sigma^2} \Big(\frac32 + \frac{\mu}{\sigma^2}\Big) \varphi_1(-\mu, \sigma^2) 
+ \frac{1}{2\sigma^4} \varphi_2(-\mu, \sigma^2)\,, \nonumber \\
\label{mom4}
\mathbb{E}[X^4] &= 
\Big( 1 + \frac{\mu^3}{6\sigma^6} - \frac{\mu}{2\sigma^4}
+ \frac{\mu^2}{\sigma^4} - \frac{1}{\sigma^2} + \frac{11\mu}{6\sigma^2}
\Big) 
\varphi(-\mu, \sigma^2) \\
&+ 
\Big( \frac{\mu^2}{2 \sigma^6} - \frac{1}{2 \sigma^4} + \frac{2 \mu}{\sigma^4} + \frac{11}{6 \sigma^2}
\Big) \varphi_1(-\mu, \sigma^2) +
\Big( \frac{\mu}{2 \sigma^6} + \frac{1}{\sigma^4} \Big) 
\varphi_2(-\mu, \sigma^2)
+ \frac{1}{6\sigma^6} \varphi_3(-\mu, \sigma^2)\,. \nonumber
\end{align}

\end{proposition}

\begin{proof}

The first moment is computed as
\begin{align}\label{firstmom}
\mathbb{E}[X]=\mathbb{E}\Big[\frac{1}{1+e^{-Z}} \Big] = 
\int_{-\infty}^\infty \frac{1}{1+e^{-\sigma z - \mu}} \phi(z) dz =
\varphi(-\mu,\sigma^2)
\end{align}
where $Z \sim N(\mu,\sigma)$ and $\phi(z) = \frac{1}{\sqrt{2\pi}} e^{-\frac12 z^2}$ is the standard normal density.
The connection between the first moment of the logit-normal random variable and the Mordell integral was noted also in \cite{Holmes2020}.

The second moment can be evaluated by taking one derivative of \eqref{firstmom} with respect to $\mu$
\begin{align}\label{dfirstmom}
\frac{d}{d\mu} \mathbb{E}[X] &=
\int_{-\infty}^\infty \frac{e^{-\sigma z - \mu}}{(1+e^{-\sigma z - \mu})^2} \phi(z) dz \\
&=
\int_{-\infty}^\infty \Big(
\frac{1}{1+e^{-\sigma z - \mu}} - \frac{1}{(1+e^{-\sigma z - \mu})^2}\Big) \phi(z) dz
= \mathbb{E}[X] - \mathbb{E}[X^2] \,. \nonumber
\end{align}
The derivative in \eqref{dfirstmom} is evaluated in terms of $\varphi, \varphi_1$ using \eqref{dphi} which gives
\begin{align}\label{X2mom}
\mathbb{E}[X^2] &= \mathbb{E}[X] - \frac{d}{d\mu} \mathbb{E}[X]\\
&= \varphi(-\mu,\sigma^2) + \varphi'(-\mu,\sigma^2) =
(1 + \frac{\mu}{\sigma^2})  \varphi(-\mu, \sigma^2) 
+ \frac{1}{\sigma^2} \varphi_1(-\mu, \sigma^2) \,. \nonumber
\end{align}

The higher moments of $X$ can be computed recursively 
using the recursion relation \cite{Holmes2020} 
\begin{equation}\label{Xmomrec}
p (\mathbb{E}[X^p] - \mathbb{E}[X^{p+1}]) = \frac{d}{d\mu} \mathbb{E}[X^p]\,.
\end{equation} 
which is proved analogously to \eqref{dfirstmom}. 
The derivatives $\partial_z\varphi_j(z,t)$ are evaluated in terms of the generalized logistic-normal integrals using the relation \eqref{dphi}.
\end{proof}

Using this recursion approach all positive integer moments $\mathbb{E}[X^p]$ can
be expressed in terms of $\varphi_j(-\mu,\sigma^2)$ with $j=0,1,\cdots, p-1$.
The numerical evaluation of the moments 
$\mathbb{E}[X^p]$ can be made either by interpolation from exact results
for $\varphi_j(-\mu,\sigma^2)$, on their respective grids, 
or using the Poisson series expansions \eqref{PoissonTheta2} to evaluate $\varphi(-\mu,\sigma^2)$ and the higher order logistic-normal integrals by finite differences
(an explicit series expansion \eqref{g1Poisson} is available for $\varphi_1(-\mu,\sigma^2)$).
We illustrate both these approaches in the next section with numerical examples. 

\section{Numerical implementation and tests}
\label{sec:6}

The series expansions for the logistic-normal integral in Proposition~\ref{propPoisson} contain Jacobi theta functions $\vartheta_j(z,q)$. These functions are available in many software languages, for example $\vartheta_j(z,q)$ is evaluated as \texttt{EllipticTheta[j,z,q]} in \textit{Mathematica},
and as \texttt{jtheta(j,z,q)} in the \texttt{mpmath} library in Python\footnote{https:/\!/mpmath.org}. Alternatively, they can be evaluated directly from their series expansions, given below in
(\ref{theta2def}) and (\ref{theta4def}).
The series expansions on the right-hand side of 
Proposition~\ref{propPoisson} 
must be truncated to a finite order. This requires some care, as we show next. 

\begin{definition}\label{def:trunc}
i) Define the truncation of the series (\ref{PoissonTheta4}) as
\begin{eqnarray}\label{PoissonTheta4N}
&& \vartheta_4(\frac{i}{2}z, e^{-\frac12 t})
g_N(z,t)  = 2S_1(z,q;N) + \frac{4\pi}{t} S_2(z,q_1;N)
\end{eqnarray}
with $q = e^{-\frac12 t}, q_1 = e^{-\frac{2\pi^2}{t}}$, and
\begin{eqnarray}
&& S_1(z,q;N) := \sum_{n=-(N-1)}^N (-1)^n e^{ (n-\frac12) z}
\frac{q^{n^2-\frac14}}{1 - q^{2n-1}} \nonumber \\
&& S_2(z,q;N) :=
\sum_{n=-N}^N  \exp\Big( \frac{2\pi in z}{t} \Big)
\frac{q^{n^2+n}}{1 + q_1^{2n}} 
= \sum_{n=-N}^N  \cos\Big( \frac{2\pi n z}{t} \Big)
\frac{q^{n^2+n}}{1 + q_1^{2n}} \nonumber
\end{eqnarray}

ii) Define the truncation of the series (\ref{PoissonTheta2}) as
\begin{eqnarray}\label{PoissonTheta2N}
\vartheta_2(\frac{i}{2}z, e^{-\frac12 t}) g_N(z,t) 
 = 
\frac{4\pi}{t} S_3(z,q_1;N) + 2 S_4(z,q;N)
\end{eqnarray}
with $q = e^{-\frac12 t}, q_1 = e^{-\frac{2\pi^2}{t}}$, and
\begin{eqnarray}
&& S_3(z,q;N) := \sum_{n=-(N-1)}^N (-1)^n e^{-\frac{2\pi i}{t} (n-\frac12) z}
\frac{q_1^{n^2-\frac14}}{1 - q_1^{2n-1}} \nonumber \\
&& = 
\sum_{n=-(N-1)}^N (-1)^n \cos\Big(\frac{2\pi }{t} (n-\frac12) z\Big)
\frac{q_1^{n^2-\frac14}}{1 - q_1^{2n-1}} \nonumber \\
&& S_4(z,q;N) := \sum_{n=-N}^N  e^{nz}
\frac{q^{n^2+n}}{1 + q^{2n}}\,.\nonumber
\end{eqnarray}
\end{definition}

This truncation preserves the symmetry properties of the infinite sums  under $
z\to -z$. In $S_1(z,q;N)$ and $S_3(z,q;N)$ the terms with 
$n=(0,1), (-1,2), \cdots$ are related by the  exchange $n\to -(n-1)$.
Keeping both terms in each pair is necessary in order to ensure
the even property in $z$ of these sums. Explicitly, we have for $S_1(z,q;N)$ 
\begin{eqnarray}
S_1(-z,q ;N) &=& \sum_{n=-(N-1)}^N (-1)^n  e^{-(n-\frac12)z}
\frac{q^{n^2-\frac14}}{1-q^{2n-1}}  \\
&=& \sum_{m=-(N-1)}^N (-1)^{1-m}  e^{(m-\frac12)z}
\frac{q^{(1-m)^2-\frac14}}{1-q^{1-2m}} \nonumber \\
&=&
\sum_{m=-(N-1)}^N (-1)^{m-1}  e^{(m-\frac12)z}
\frac{q^{m^2-\frac14}}{1-q^{1-2m}} q^{1-2m} = S_1(z,q;N)\,.\nonumber
\end{eqnarray}
In the second line we denoted $m=-(n-1)$ running from $N$ to $-(N-1)$.
 
A similar property holds for $S_2(z,q;N)$ and $S_4(z,q;N)$ where the
terms with indices $n$ and $-n$ are both required for the even property in $z$.
Truncating the sums as shown above preserves this symmetry.

\begin{figure}[h]
    \centering
   \includegraphics[width=2.5in]{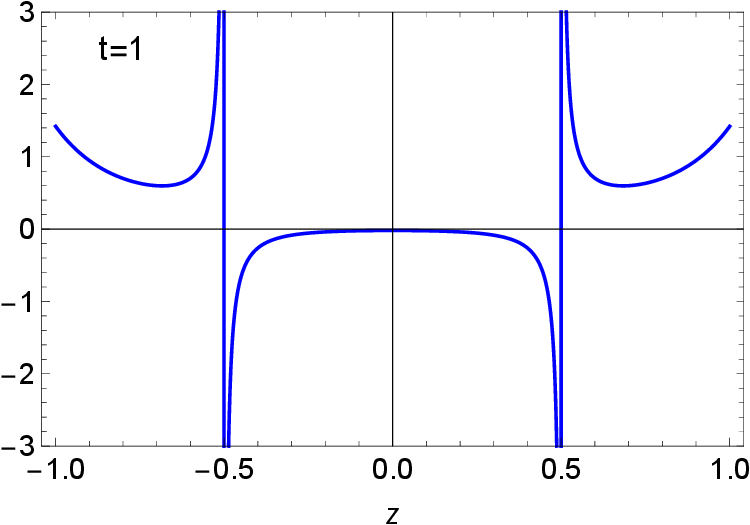}
   \includegraphics[width=2.5in]{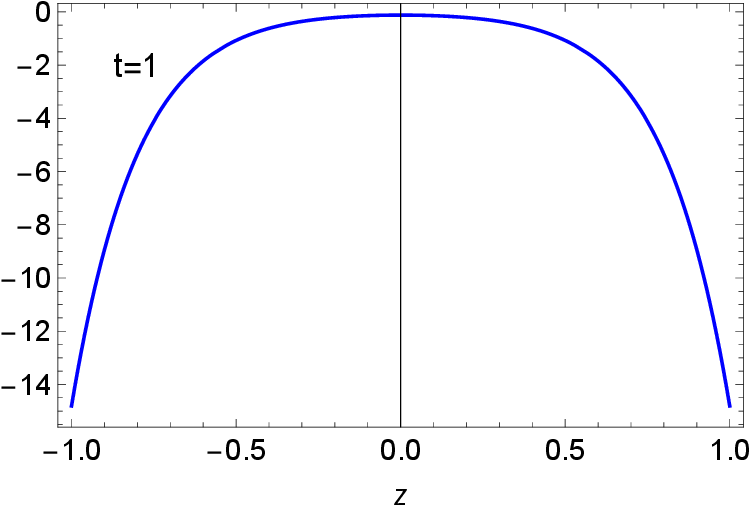}
    \caption{
The approximation error $\Delta_N g(z,t) =g_N(z,t) - g(z,t)$
of the series (\ref{PoissonTheta4N}) (left) 
and (\ref{PoissonTheta2N}) (right) vs $z$ at $t=1$, truncated with
$N=5$. The benchmark for $g(z,t)$ is trapezoidal quadrature (\ref{Trapez})
with step $h=0.1$ and $N_q=100$. Left: $10^4 \cdot \Delta_N g(z,t)$ and right:
$10^8 \cdot \Delta_N g(z,t)$.
}
\label{Fig:1}
 \end{figure}

The relative numerical performances of the two series
(\ref{PoissonTheta4N}) and (\ref{PoissonTheta2N}) are compared in 
Figure~\ref{Fig:1} for $t=1$. 
The benchmark for this test is trapezoidal quadrature with step $h$
\begin{equation}\label{Trapez}
g(z,t) = h \sum_{k=-N_q}^{N_q} \frac{1}{\sqrt{2\pi t}} e^{-\frac{1}{2t}(z-hk)^2}
\frac{1}{\cosh(h k z/2)}
+ E(h)
\end{equation}
The quadrature step was chosen $h=0.1$, which gives the error 
bound $|E(h)| \leq 10^{-42}$ for all $t>1$, see \cite{CS,DP} for the error bound expression.

The left plot in Figure~\ref{Fig:1} shows that the series 
(\ref{PoissonTheta4N})
introduces noise around the points $z_k= (k+\frac12)t$ due to the $0/0$
phenomenon noted above. On the other hand, the evaluation using
the series (\ref{PoissonTheta2N}) is much more stable for values sufficiently 
close to the origin.


The numerical evaluation of (\ref{PoissonTheta2N}) for $|z| \gg t/2$ 
involves the division of two very small numbers, which introduces large 
errors. This can be avoided by computing $g(z,t)$ for 
$z\in (-\frac12 t, \frac12 t)$ using (\ref{PoissonTheta2N}), and 
defining $g(z,t)$ outside this interval
by repeated application of the relation (\ref{Mordell2g}). The 
error of the resulting approximation for $g(z,t)$ is bounded by the following 
result, see \cite{DP}.

\begin{proposition}\label{properr}
Denote $\bar g(z,t)$ any approximation of $g(z,t)$ defined on the interval 
$z:(-\frac12 t, \frac12 t)$, and by repeated application of (\ref{Mordell2g}) 
outside of this interval.
The approximation error $\Delta g(z,t) = g(z,t) - \bar g(z,t)$ is bounded as
\begin{eqnarray}\label{errbound}
|\Delta g(z,t)| \leq e^{-\frac{1}{2t} z^2 + \frac18 t} 
\sup_{-\frac12 t\leq q\leq \frac12 t} |\Delta g(z,t)| \,.
\end{eqnarray}
\end{proposition}

Assume that the approximation error in $x\in [-t/2,t/2]$ is below a 
prescribed level $\varepsilon$.
Then Proposition~\ref{properr} bounds the error of the resulting approximation 
for $g(x,t)$ for any $x\in \mathbb{R}$ as $|\Delta g(z,t)| \leq 
\min(\varepsilon e^{-\frac{1}{2t}z^2+\frac18 t}, \varepsilon)$. 
In Appendix~\ref{sec:C} we
present an explicit algorithm implementing this method. A similar method
was used in \cite{GSM} for the evaluation of the logistic-normal integral 
$\varphi(z,t)$, using the approximation of \cite{MonahanStefanski}
as starting point in the primitive cell $(-t/2,+t/2)$.

Figure~\ref{Fig:2} shows the truncation error 
$\log_{10}|g_N(z,t) - g_{\rm T}(z,t)|$ computed also with respect to a
benchmark evaluation using trapezoidal quadrature (\ref{Trapez}).
The error decreases rapidly with the truncation order $N$.
\begin{figure}[h]
\centering
\includegraphics[width=2.5in]{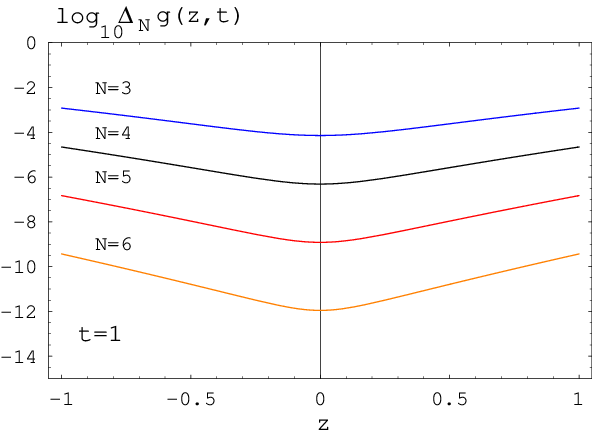}
\caption{
The truncation error $\log_{10}|g_N(z,t) - g_T(z,t)|$ 
of the series (\ref{PoissonTheta2N}) 
at $t=1$ evaluated by truncation with $N=3,4,5,6$.
The benchmark is trapezoidal quadrature (\ref{Trapez})
with step $h=0.1$.}
\label{Fig:2}
 \end{figure}

\begin{figure}[h]
    \centering
   \includegraphics[width=2.5in]{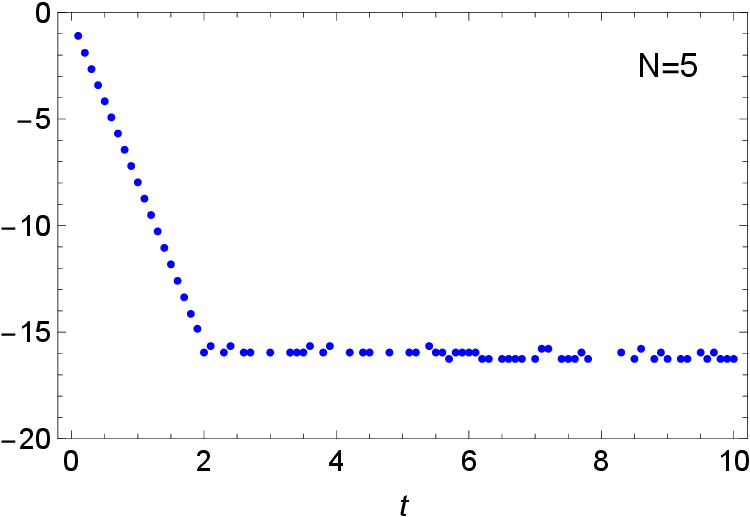}
   \includegraphics[width=2.5in]{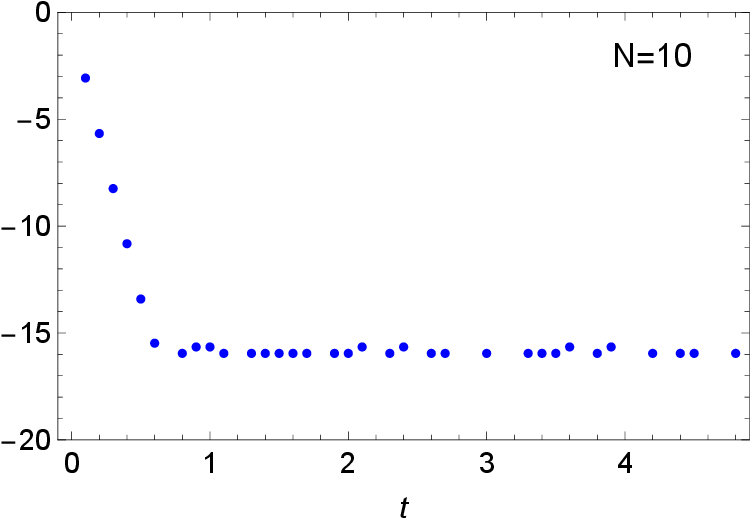}
    \caption{
The truncation error $\log_{10}|g_N(\frac12 t,t) - e^{-\frac18 t}|$ 
of the series (\ref{PoissonTheta2N}) 
vs $t$ truncating the series to $N=5$ terms (left) and $N=10$ (right).
}
\label{Fig:3}
 \end{figure}

The truncation error of the series (\ref{PoissonTheta2N}) can be estimated
by comparing the evaluation of this series at $x=\frac12 t$ against the exact result
$g(\frac12 t,t) = e^{-\frac18 t}$. The results are shown in Figure~\ref{Fig:3},
which plots $\log_{10}|g_N(z=\frac12 t,t)-e^{-\frac18 t} |$.
The error decreases with $t$ and approaches a floor at about $\sim 10^{-16}$,
corresponding to machine precision (double-precision, floating point
accuracy in \textit{Mathematica}).
Truncating the series at $N=10$ terms gives an approximation
error of the order of $\sim 10^{-16}$ over a wide range of $t$. However, for 
small $t < 0.5$, the Poisson series expansion
converges more slowly and the truncation error becomes large as $t$ decreases.

In the small-$t$ region we propose to evaluate $\varphi(z,t)$ by interpolation
from the exactly known values on the grid $z_k=k t$, see Proposition~\ref{prop:2}.
As $t\to 0$, the grid step decreases to zero and the interpolation is expected to become very precise. 
A similar approach can be used for the higher order integrals $\varphi_j(z,t)$, which can be evaluated on their respective grids, as shown in Table~\ref{tab:1}.
The exact evaluations for $\varphi_j(z,t)$ with $j=1,2,3$ are given in Propositions~\ref{prop:phi1},
\ref{prop:phi2} and \ref{prop:phi3}, respectively.

We illustrate the application of the interpolation method to the evaluation of first two moments of a logistic-normal 
random variable $X \sim \mbox{logitnorm}(\mu, \sigma)$. They are related to the logistic-normal integrals $\varphi, \varphi_1$ as in \eqref{mom1} and \eqref{mom2}, respectively.

Numerical evaluations of $\mathbb{E}[X]$ and $\mathbb{E}[X^2]$ 
for $\mu = 1$ and several values of $\sigma$ 
are shown in Table~\ref{tab:23}. The interpolation method is applied separately to $\varphi(z,t)$ and $\varphi_1(z,t)$ and the results are combined. For evaluating $\varphi(-\mu,\sigma^2)$, first find $k\in \mathbb{Z}$ such that 
$k \sigma^2 \leq -\mu \leq (k+1)\sigma^2$ and then interpolate $\varphi(-\mu,\sigma^2)$ from the exactly known values at the ends of the interval. A similar approach is used to obtain $\varphi_1(-\mu,\sigma^2)$ from its exactly known values $\varphi_1((k+\frac12) \sigma^2, \sigma^2)$ and 
$\varphi_1((k+\frac32) \sigma^2,\sigma^2)$ with $k\in \mathbb{Z}$. For simplicity we use linear interpolation. 
We report also the error bound\footnote{Denote $\bar f(x)$ the linear interpolation of a function $f(x)$ from its values at the end points of the interval $[a,b]$. 
The interpolation error is bounded as $|\bar f(x) - f(x)| \leq 
\delta :=\frac18 M_2 (b-a)^2$ with $M_2 = \sup_{a\leq x\leq b} |f''(x)|$ the supremum of the second derivative taken over the interpolation domain. } 
of the linear interpolation $\delta$.

The evaluation using the series expansion with fixed truncation order $N=12$ breaks down for small $\sigma$, while the evaluation by interpolation from the exactly known values gives very precise results. 
 These results suggest a hybrid approach for the numerical evaluation of the moments: 
for sufficiently small $\sigma<0.1$ use the interpolation method, while for $\sigma>0.1$ use the series expansion method based on \eqref{PoissonTheta2} and extrapolation based on the recursion method of Proposition~\ref{properr}.


\begin{table}[h!]
\begin{center}
\caption{Numerical testing for the first two moments of a logistic-normal random variable
$X\sim \mbox{logitnorm}(\mu,\sigma)$ with $\mu=1$ and several values of $\sigma$. The table shows the results obtained from the Poisson series expansion truncated to $N=12$ terms, linear interpolation from the grid of exactly known values, the bound on the interpolation error, and 
the result obtained from numerical integration. 
} \label{tab:23}
{\begin{tabular}{|l|cccc|}
\hline
 & \multicolumn{4}{|c|}{$\mathbb{E}[X]$}  \\
 \hline
$\sigma$	& Poisson $(N=12)$ & Interpolation & Interp Err $\delta$ & Num Int \\ 
\hline
1.00 & 0.696735 & 0.696735 & 0.00 & 0.696735 \\
0.50 & 0.720581 & 0.720581 & 0.00 & 0.720581 \\
0.10 & NaN & 0.730606 & 1.1E-06 & 0.730606 \\
0.05 & NaN & 0.730945 & 7.1E-08 & 0.730945 \\
0.01 & NaN & 0.731054 & 0.00 & 0.731054 \\
\hline
0 & & & & $1/(1+e^{-1})= 0.731059$ \\
\hline
\hline
 & \multicolumn{4}{|c|}{$\mathbb{E}[X^2]$}  \\
 \hline
$\sigma$	& Poisson $(N=12)$ & Interpolation & Interp Err $\delta$ & Num Int \\ 
\hline
\hline
1.00 & 0.518789 & 0.582174 & 0.031752 &  0.518791 \\
0.50 & 0.528836 & 0.551173 & 0.019577 & 0.528640 \\
0.10 & NaN & 0.535193 & 0.001329 & 0.534171 \\
0.05 & NaN & 0.534634 & 0.000336 & 0.534377 \\
0.01 & NaN & 0.534454 & 0.000010 & 0.534444 \\
\hline
0 & & & & $1/(1+e^{-1})^2= 0.534447$ \\
\hline
\end{tabular}}
\end{center}
\end{table}

\section{Summary}

The logistic-normal integral and its derivatives satisfy a large number of symmetry relations which lead to exact evaluations and series expansions. While one particular series has been presented in the literature, see \cite{Mordell1933} and \cite{Johnson}, it is less appreciated that there exists a continuum of such series, with different stability properties under numerical evaluation.
We point out that there exists an optimal choice in this continuum which is best suited for numerical evaluation, and give practical details for its application. 

We discuss also a class of generalizations of the logistic-normal integral, which are relevant for the computation of its derivatives. The logistic-normal integral and its generalizations can be evaluated in closed form on grids of uniformly spaced points. These evaluations follow from recursion relations satisfied by these integrals. 
We give explicit solutions of these recursions which can be used for efficient evaluation of the exact values. The exact values on the grids can be used to construct interpolations for the logistic-normal integral and its generalizations. 

As an application we discuss the evaluation of the first few moments of a logistic-normal random variable $X \sim \mbox{logitnorm}(\mu,\sigma)$, which can be expressed in terms of the logistic-normal integral and its generalizations.
We present numerical tests for the numerical evaluation of the first two moments and compare the performance of the series expansions and of the interpolation from exactly known values. These two methods complement each other in different regimes of the parameter $\sigma$: for small $\sigma$ the interpolation method is optimal, while for large $\sigma$ the series expansion 
converges rapidly and gives precise results.

\appendix


\section{Relation to the Mordell integral}
\label{sec:app.A}

The logistic-normal integral is related to the Mordell integral, 
which was introduced in the context of 
analytical number theory by L.~J.~Mordell in \cite{Mordell1920,Mordell1933}.
Following the notation of \cite{Zwegers}, it can be defined as follows, 
with $z \in \mathbb{C}$ and
$\tau \in \mathcal{H}$ and
$\mathcal{H} = \{z = x+ iy; y > 0\}$
\begin{eqnarray}\label{Mordelldef}
h(z;\tau)=\int_{-\infty}^\infty dx \frac{e^{i\pi \tau x^2 - 2\pi z x}}
{\cosh \pi x}\,.
\end{eqnarray}
This integral plays an important role in the theory of the modular
forms, and has been studied extensively in relation to the mock theta
functions in \cite{Andrews,Zwegers,CR}. 

The function $g(x,t)$ is related to the Mordell integral of  imaginary $\tau$ as
\begin{eqnarray}\label{g2h}
g(z,t) = \sqrt{\frac{2\pi}{t}} e^{-\frac{1}{2t}z^2} 
h\left(\frac{z}{t}, i\frac{2\pi}{t}\right)\,.
\end{eqnarray}

The Mordell integral satisfies a large number of symmetry relations 
which were proved in \cite{Mordell1933}. We list a subset of these relations below, 
following the notations of the Proposition 1.2 in \cite{Zwegers}.

\begin{proposition}[\cite{Mordell1933}]
The function $h(z;\tau)$ is doubly-quasiperiodic in the $z$ argument, with 
periods $(1,\tau)$
\begin{eqnarray}\label{Mprop1}
\mbox{(1)} && \qquad h(z;\tau) + h(z+1;\tau) = \frac{2}{\sqrt{-i\tau}} 
e^{\frac{\pi i}{\tau}(z+\frac12)^2} \\
\label{Mprop2}
\mbox{(2)} && \qquad h(z;\tau) + e^{-2\pi i z - \pi i \tau}h(z+\tau;\tau) = 
2e^{-\pi i z - \pi i \tau/4} \,.
\end{eqnarray}

(3) $z\to h(z;\tau)$ is the unique holomorphic function satisfying (1) and (2).

It is an even function of the first argument 
\begin{eqnarray}\label{Mordell4}
\mbox{(4)} && \qquad h(-z;\tau) = h(z;\tau)\,.
\end{eqnarray}

Under modular transformations in the second argument it changes as
\begin{eqnarray}\label{Mordell5}
\mbox{(5)} && \qquad h\left(\frac{z}{\tau}; -\frac{1}{\tau}\right) = 
\sqrt{-i\tau} e^{-\pi i \frac{z^2}{\tau}} h(z; \tau)\,.
\end{eqnarray}
\end{proposition}

It is easy to see that the symmetry relation (1) is equivalent to the 
relation (\ref{Mordell2g}), and the relation (5) is equivalent to 
(\ref{Mordell5g}).

\section{Proofs}
\label{sec:app.B}

\begin{proof}[Proof of Proposition~\ref{prop:2}]
Consider the alternating geometric sum
\begin{eqnarray}\label{geomsum}
\sum_{k=0}^n (-1)^k e^{kx} = \frac{1 - (-1)^{n+1}e^{(n+1)x}}{1 + e^x} =
\frac12 e^{-\frac12 x} \frac{1 - (-1)^{n+1}e^{(n+1)x}}{\cosh(x/2)} \,.
\end{eqnarray}
Multiply both sides with $\frac{1}{\sqrt{2\pi t}} e^{-\frac{1}{2t}x^2}$ 
and integrate over $x$. This gives
\begin{eqnarray}
\sum_{k=0}^n (-1)^k e^{\frac12 k^2 t} = \frac12 e^{\frac18 t}g(-\frac12 t, t) -
(-1)^{n+1} \frac12 e^{\frac12(n+\frac12)^2 t} g\Big(( n + \frac12) t, t\Big)\,.
\end{eqnarray}
Using $g\left(-\frac12 t, t\right) = e^{-\frac18 t}$, see (\ref{gatt2}), gives
\begin{eqnarray}
\sum_{k=0}^n (-1)^k e^{\frac12 k^2 t} = \frac12 -
(-1)^{n+1} \frac12 e^{\frac12(n+\frac12)^2 t} g\Big( (n + \frac12) t, t\Big)\,,
\end{eqnarray}
and thus
\begin{align}\label{GaussSum10}
g\Big((n+\frac12)t, t\Big) &= 
2 (-1)^n e^{-\frac12(n+\frac12)^2 t} \Big(
\sum_{k=0}^n (-1)^k e^{\frac12 k^2t} - \frac12 \Big) \\
&= 2 \Big( \sum_{k=0}^n (-1)^{n-k} e^{\frac12 (k^2-n^2-n-\frac14)t} -
\frac12 (-1)^n  e^{-\frac12 (n+\frac12)^2 t} \Big) \nonumber \\
&= 2 \Big( e^{-\frac12 (n + \frac14)t} \sum_{j=0}^n (-1)^{j} e^{(\frac12 j^2 - jn)t}  -
\frac12 (-1)^n  e^{-\frac12 (n+\frac12)^2 t} \Big) \nonumber 
\end{align}
where we introduced $j=n-k$ in the last line. 
By \eqref{varphi2g} we have $\varphi((n+1)t,t) = \frac12 e^{-\frac12(n+\frac12)t} g((n+\frac12 )t,t)$
such that we get exact evaluations for $\varphi( nt,t)$ with $n\geq 1$. The $n<0$ case is obtained 
using that $g$ is an even function in $z$, which gives $\varphi(-z,t) = 1 - \varphi(z,t)$.

\end{proof}

\begin{proof}[Proof of Proposition~\ref{prop:1}]

This relation follows from the fact that both factors in the
definition of $g(z,t)$, the Gaussian function and the $(\cosh x)^{-1}$ function, 
have the property that they have the same functional form as their own
Fourier transforms. The Fourier transform of $g(z,t)$ is
\begin{eqnarray}
\tilde g(\omega,t) = \int_{-\infty}^\infty e^{i\omega z} g(z,t) = 
\frac{2\pi}{\cosh \pi \omega } e^{-\frac12 \omega^2 t}\,.
\end{eqnarray}
Taking the inverse Fourier transform we have
\begin{eqnarray}
g(y,t) &=& \int_{-\infty}^\infty d\omega e^{-i\omega y} \tilde g(\omega,t) \\
&=& \int_{-\infty}^\infty \frac{d\omega}{\cosh \pi\omega} 
    e^{-\frac12\omega^2 t-i\omega y} \nonumber \\
&=& \frac{1}{2\pi} e^{-\frac{1}{2t}y^2} \int_{-\infty}^\infty \frac{du}
{\cosh(u/2)} e^{-\frac{1}{8\pi^2}(u+2\pi i y)^2} \nonumber \\
&=& \sqrt{\frac{2\pi}{t}} e^{-\frac{1}{2t} y^2} 
g\left( \frac{2\pi y}{it},\frac{4\pi^2}{t}\right)\,.
\nonumber
\end{eqnarray}
This reproduces the relation (\ref{Mordell5g}).

\end{proof}


\begin{proof}[Proof of Proposition~\ref{propPoisson}]
We follow the same approach as in Section 6 in \cite{DP}. The starting point is the Fourier transform of $g(z,t)$ 
\begin{equation}
g(z,t) = \frac{1}{2\pi} \int_{-\infty}^\infty e^{-i\omega z} \tilde g(\omega,t) d\omega
\end{equation}
where $\tilde g(\omega,t)$ was given above in the proof of Proposition \ref{prop:1}
\begin{equation}
\tilde g(\omega,t) = \frac{2\pi}{\cosh \pi \omega} e^{-\frac12 \omega^2 t}\,.
\end{equation}

By the Poisson summation formula we have
\begin{equation}
\sum_{k=-\infty}^\infty g(z+k t,t) = \frac{1}{t} \sum_{n=-\infty}^\infty \tilde g \Big( \frac{2\pi n}{t},t\Big)
e^{\frac{2\pi i n}{t} z}
\end{equation}
All the terms in the sum on the left hand side can be expressed in terms of $g(z,t)$ by repeated application of the recursion relation \eqref{Mordell2g}. Thus we can use this expression to express 
$g(z,t)$ in terms of the sum on the right hand side.
The result can be put into the form
\begin{eqnarray}\label{4.5}
&& \frac12 e^{\frac18 t} \vartheta_4(\frac{i}{2}z, e^{-\frac12 t})
g(z,t) \\
&& = \sum_{j=1}^\infty (-1)^j \frac{\cosh[\frac12 (2j-1)z] \exp[\frac12 (j-j^2)t]}
{\sinh[\frac14 (2j-1) t]}  + \frac{\pi}{t}e^{\frac18 t}
\sum_{k=-\infty}^\infty e^{-\frac{2\pi^2}{t} k^2}
\frac{\cos(\frac{2\pi z k}{t})}
{\cosh (\frac{2\pi^2 k}{t})} \nonumber
\,. 
\end{eqnarray}
where $\vartheta_4(z,q)$ is
one of the Jacobi theta functions given by \cite{WW}
\begin{eqnarray}\label{theta4def}
\vartheta_4(z,q)= \sum_{n=-\infty}^\infty (-1)^n q^{n^2} \exp(2 n i z) 
=1 + 2 \sum_{n=1}^\infty (-q)^{n^2} \cos(2nz)\,.
\end{eqnarray}

The two sums appearing in Eq.~(\ref{4.5}) can be expressed as
\begin{eqnarray}\label{S1def}
S_1 &:=& \sum_{j=1}^\infty (-1)^j \frac{\cosh[\frac12 (2j-1)z] 
\exp[\frac12 (j-j^2)t]}
{\sinh[\frac14 (2j-1) t]} \\
 &=&
2 e^{\frac14 t} \sum_{j=1}^\infty (-1)^j \cosh[ (j-\frac12) z]
\frac{q^{j^2}}{1 - q^{2j-1}} \nonumber \\
 &=&
e^{\frac14 t} \sum_{j=-\infty}^\infty (-1)^j e^{ (j-\frac12) z}
\frac{q^{j^2}}{1 - q^{2j-1}} \,,\quad q = e^{-\frac12 t} \nonumber
\end{eqnarray}
and
\begin{eqnarray}\label{S2def}
S_2 &:=& 
\sum_{k=-\infty}^\infty e^{-\frac{2\pi^2}{t} k^2}
\frac{\cos(\frac{2\pi z k}{t})}
{\cosh (\frac{2\pi^2 k}{t})} 
 =
2 \sum_{j=-\infty}^\infty  \cos\Big( \frac{2\pi j z}{t} \Big)
\frac{q_1^{j^2+j}}{1 + q_1^{2j}} \\
&=& 2\sum_{j=-\infty}^\infty  \exp\Big( \frac{2\pi ij z}{t} \Big)
\frac{q_1^{j^2+j}}{1 + q_1^{2j}}\,,\quad q_1=e^{-\frac{2\pi^2}{t}}\,. \nonumber
\end{eqnarray}

This concludes the proof of the result (\ref{PoissonTheta4}).
The series (\ref{PoissonTheta2}) is obtained from (\ref{PoissonTheta4})
by an application of one of the Jacobi identities for the theta 
functions
\begin{eqnarray}\label{Jacobi4}
\vartheta_4\left(i\frac{z}{2},e^{-\frac12 t}\right) = \sqrt{\frac{2\pi}{t}}
\exp\Big(\frac{z^2}{2t}\Big)
\vartheta_2\left(\frac{\pi z}{t},e^{-\frac{2\pi^2}{t}}\right)\,,
\end{eqnarray}
and using the relation (\ref{Mordell5g}) for $g(z,t)$. The Jacobi theta
function $\vartheta_2(z,q)$ is defined as
\begin{eqnarray}\label{theta2def}
\vartheta_2(z,q)= 2 q^{1/4} \sum_{n=0}^\infty  q^{n^2+n} \cos((2 n + 1) z) \,.
\end{eqnarray}

\end{proof}


\section{Algorithm}
\label{sec:C}

We give in this Appendix a detailed implementation of the algorithm
for the approximation $\bar g(z,t)$ described in Proposition~\ref{properr}.
The algorithm takes as input \textit{gMain(x,t)} which can be computed 
using the series (\ref{PoissonTheta2N}). 

An implementation in R of this algorithm, together with an evaluation of $g(z,t)$ using (\ref{PoissonTheta2}) in the primitive cell $z\in [-t/2,t/2]$ is available at
\texttt{https://github.com/dan-pirjol/logisticNormal}

\begin{algorithm}
\label{alg:1}
\caption{Algorithm for computing the integral $g(x,t)$ by recursion
using Eq.~(\ref{Mordell2g}) from its values in the primitive cell
$z:(-\frac12t,+\frac12t)$.
Takes as input \textit{gMain(x,t)}.}
\label{LNalgo}
\begin{algorithmic}[1]
\STATE $x = \mbox{abs}(x)$

\STATE $k = \mbox{int}(x/t)$
\STATE $x0 = x - k*t$ 

\IF{$x0 > 0.5*t$}

\STATE $x0 = x0 - t$
\STATE $k = k+1$

\ENDIF

\STATE $g0 = \mbox{gMain}(x0,t)$

\STATE $z = x0$
\STATE $g = g0$

\FOR{$j=1$ to $k$}
  \STATE $g = 2*\mbox{exp}(-0.5*z - 0.375*t) - \mbox{exp}(- z - 0.5*t)*g$
  \STATE $z = z + t$
\ENDFOR

\RETURN $g$
\end{algorithmic}
\end{algorithm}


\bibliographystyle{plain}
\bibliography{MordellPaper.bib}
\end{document}